\documentclass[11pt]{amsart}
\usepackage{amssymb,latexsym}

\theoremstyle{definition}

\newtheorem{theorem}{Theorem}

\newtheorem{lemma}[theorem]{Lemma}
\newtheorem{definition}[theorem]{Definition}

\newtheorem{notation}[theorem]{Notation}
\newtheorem{remark}[theorem]{Remark}

\date{}
\newcommand{\numberset}{\mathbb}

\newcommand{\F}{\numberset{F}}

\newcommand{\Pro}{\numberset{P}}

\newcommand{\Ol}{\mathcal{O}}

\newcommand{\mC}{\mathcal{C}}

\usepackage{hyperref}

\usepackage[margin=2.4cm]{geometry}

\usepackage{mathptmx}

\usepackage{amsaddr}

\begin{document}

\title[]{A zero-dimensional approach to Hermitian codes}

\author{Edoardo Ballico$^1$}
\address{Department of Mathematics, University of Trento\\Via Sommarive 14,
38123 Povo (TN), Italy}
\email{$^1$edoardo.ballico@unitn.it}

\author{Alberto Ravagnani$^2$ $^*$}
\address{Department of Mathematics, University of Neuch\^{a}tel\\Rue
Emile-Argand 11, CH-2000 Neuch\^{a}tel, Switzerland}
\email{$^2$alberto.ravagnani@unine.ch}

\thanks{$^1$Partially supported by MIUR and GNSAGA}
\thanks{$^*$Corresponding author}

\maketitle

\providecommand{\bysame}{\leavevmode\hbox to3em{\hrulefill}\thinspace}

\begin{abstract}

We study the algebraic geometry of a family of evaluation codes from plane
smooth 
curves defined over any field. In particular, we provide a
cohomological
characterization of their dual minimum distance. 
After having discussed some general results on zero-dimensional subschemes
of
the plane,  we focus
on the interesting case of Hermitian $s$-point codes, describing the geometry of
their dual minimum-weight codewords.
\end{abstract}

\section{Introduction}\label{S1}
Let $\F$ be any finite field and let $n \ge 1$ be an integer. A \textbf{linear
code} of length $n$ and
dimension $k$ over $\F$ is a $k$-dimensional vector subspace $\mathcal{C}
\subseteq \F^n$.
 The elements of $\mathcal{C}$ are called \textbf{codewords}. For any
 $v,w \in \mC$ define the \textbf{distance} between $v$ and $w$ by
$d(v,w):=|\{ 1 \le i \le n : v_i \neq w_i\}|$. The \textbf{weight} of a codeword
$v \in \mC$ is defined
as $\mbox{wt}(v):=d(v,0)$. The \textbf{minimum distance} of a code $\mC
\subseteq \F^n$
of at least two elements is the positive integer
$$d(\mC):=\min_{v \in \mC \setminus \{ 0 \}} \mbox{wt}(v)=\min_{v \neq w \in
\mC} d(v,w).$$ A code of minimum distance $d$ corrects
 $\lfloor (d-1)/2 \rfloor$ errors: the higher is the
minimum distance, the
higher is the correction capability. Define the component-wise product in
$\F^n$ by $v\cdot w:=\sum_{i=1}^n v_iw_i$.
The \textbf{dual code} of $\mC$ is the $(n-k)$-dimensional code
 $\mC^\perp:=\{ u \in \F^n : u \cdot v=0 \mbox{ for any } v \in \mC\}$.

\begin{definition}\label{siso}
 We say that codes $\mC, \mathcal{D} \subseteq \F^n$ are \textbf{strongly
isometric} if the codewords of $\mC$ are obtained multiplying component-wise
 the codewords of $\mathcal{D}$  by a vector of $\F^n$ whose
components are all
nonzero.
\end{definition}

\begin{remark}
The strong isometry is an equivalence relation on the set of codes in $\F^n$.
Strongly isometric codes have the same dimension and 
the same minimum distance. Moreover, a strong isometry preserves the support of
the codewords and, in particular, the number of minimum-weight codewords of a
code. Codes $\mC$ and
$\mathcal{D}$ are strongly isometric if and only if their dual codes $\mC^\perp$
and $\mathcal{D}^\perp$ are strongly isometric. 
\end{remark}

 Let $\Pro^r$ be the projective $r$-dimensional space over
 $\F$, and let $C \subseteq \Pro^r$ be a connected smooth curve
defined\footnote{The
curve $C$ could be not defined over $\F$, but only over the algebraic closure
$\overline{\F}$ of $\F$.} over $\F$. Assume that $C$ is a
complete intersection. Choose any subset $B \subseteq C(\F)$ of $\F$-rational
points of $C$ and 
 an integer $d>0$. Finally, consider the linear map
$$\mbox{ev}: H^0(C,\Ol_C(d)) \to \F^{|B|}$$ ($|B|$ denotes the cardinality
of $B$) which evaluates a degree $d$ homogeneous form on $C$ at the points
appearing in $B$. Being a vector subspace of $\F^{|B|}$, the image of
$\mbox{ev}$, say $\mathcal{C}$, is a linear code of
length $|B|$ over the finite field $\F$.

Recently, A. Couvreur showed in \cite{c} that a lower bound on the minimum
distance
 of $\mathcal{C}^\perp$ can be expressed in terms of $d$ and the projective
geometry of $B$
 (for instance, the existence in $B$ of $d+2$ collinear points).
 Codes arising from geometric constructions are known to have good parameters
for
 applications and a wide literature on the topic is available
 (see in particular \cite{Ste} and \cite{Sti}).

In this paper we focus on the case  $r=2$ of the described approach (i.e., on
the
case of plane smooth curves) and provide
an improvement of Couvreur's method in this specific context. More precisely,
we introduce zero-dimensional schemes in the setup, and study codes obtained
evaluating vector spaces of the more general form $H^0(C,\Ol_C(d)(-E))$,
where $E \subseteq \Pro^2$ is a zero-dimensional scheme  whose support avoids
the set $B$ in the notation above. This class of codes includes many classical
Goppa codes (see Remark \ref{goppa} at
page \pageref{goppa}). Then we apply the results
for arbitrary curves
to the special case of codes from the Hermitian curve, providing a geometric
characterization of the dual minimum-weight codewords of many Hermitian
$s$-point codes
(see \cite{Ste}, Chapter 10 for the definitions).

 \subsection{Layout of the paper}

The paper is organised in three main parts.
In Section \ref{general} we characterize the dual minimum distance
of codes
arising from smooth plane curves, and establish a key lemma 
 to control zero-dimensional plane schemes 
from a cohomological point of view. In Section \ref{hhh} we describe some
geometric properties of Hermitian $s$-point codes. In Section \ref{hhh2} we
prove our main results on
Hermitian $s$-point codes.

\subsection{Main references}
One-point codes from the Hermitian curve
are well-studied, and efficient methods to decode them are known
(\cite{Sti},\cite{yk} and \cite{yks}).
The minimum distance of Hermitian two-point codes has been first determined by
M. Homma and S. J. Kim
 (\cite{hk1}, \cite{hk2}, \cite{hk3}, \cite{hk4}) and more recently S. Park gave
explicit formulas
for the dual minimum distance of such codes (see \cite{p}) using different
techniques.
The second and the
third Hamming weight of one-point codes on the Hermitian curve are studied in
\cite{m}, \cite{yks} and
\cite{mr}. The second
Hamming weight of Hermitian two-point codes is treated in \cite{hk5}.

\section{Results on arbitrary plane smooth curves} \label{general}
In this section we study the algebraic geometry of 
codes obtained 
evaluating vector spaces of homogeneous forms at some prescribed points of a
plane smooth
curve.
The general results of this section will be applied to
interesting examples
in later sections.

\begin{remark}
 The definition of linear code (see Section \ref{S1}) can be given over any
field $\F$ (and not only over finite fields). The minimum distance is defined in
the same way. We use this extended definition in this section.
\end{remark}

\begin{remark}\label{a000.000}
Fix any field $\F$, a projective scheme $T$ and a coherent sheaf $\mathcal{F}$
on $T$ defined over $\F$.
Let $\overline{\F}$ be the algebraic closure of  $\F$.
 Consider the scheme $T_{\overline{\F}}$
and the coherent sheaf $\mathcal {F}_{\overline{\F}}$ obtained by the extension
 of scalars $\overline{\F} \supseteq \F$. 
 The definition and main properties of cohomology groups
can be found in \cite{ha}, Chapter II and III, for arbitrary
schemes and over an arbitrary field.
Since every extension of fields $\overline{\F} \supseteq \F$ is flat,
for each integer $i\ge 0$ the $\F$-vector space $H^i(T,\mathcal {F})$ and the
$\overline{\F}$-vector
space $H^i(T_{\overline{\F}},\mathcal {F}_{\overline{\F}})$ have the same
dimension (see \cite{ha}, Proposition III.9.3).
\end{remark}

\begin{lemma}\label{e1}
Let $\F$ be any field and let $\Pro^2$ denote the projective plane on $\F$.
Let $C \subseteq \Pro^2$ be a smooth plane curve. Fix an integer $d>0$, a
zero-dimensional
scheme $E\subseteq C$ and a finite
subset $B\subseteq C$ such that $B\cap E_{red}=\emptyset$\footnote{Here
$E_{red}$ denotes the reduction of the scheme $E$.}. Denote by $\mathcal{C}$ the
code 
obtained evaluating the vector space $H^0(C,\Ol_C(d)(-E))$ at the points of $B$.
Set $c:= \deg (C)$,
$n:=|B|$ and assume $|B| > dc -\deg (E)$. The following facts hold.
\begin{enumerate}
 \item The dimension of $H^0(C,\mathcal{O}_c(d))$ is given by the formulas
$$h^0(C,\mathcal {O}_C(d)) = \left\{ \begin{array}{ll} \binom{d+2}{2} & \mbox{
if } d<c, \\
                                  \binom{d+2}{2} -\binom{d-c+2}{2} & \mbox{ if }
d \ge c. \end{array} \right.\ $$
\item The code $\mathcal {C}$ has length $n$ and dimension $k:= h^0(C,\mathcal
{O}_C(d))-\deg (E) +h^1(\mathbb {P}^2,\mathcal {I}_E(d))$.

\item The minimum distance
of $\mathcal {C}^{\bot}$ is the minimal cardinality, say $z$, of a subset of $S
\subseteq B$
such that
$$h^1(\mathbb {P}^2,\mathcal {I}_{S\cup E}(d)) >h^1(\mathbb {P}^2,\mathcal
{I}_E(d)).$$

\item A codeword of $\mathcal{C}^{\bot}$
has weight $z$ if and only if it is supported by a subset $S\subseteq B$ such
that
\begin{enumerate}
 \item $|S| = z$,
\item $h^1(\mathbb {P}^2,\mathcal {I}_{E\cup S}(d)) >h^1(\mathbb
{P}^2,\mathcal {I}_E(d))$,
\item $h^1(\mathbb {P}^2,\mathcal {I}_{E\cup S}(d))
>h^1(\mathbb {P}^2,\mathcal {I}_{E\cup S'}(d))$ for any $S'\varsubsetneq S$.
\end{enumerate}
\end{enumerate}
\end{lemma}

\begin{proof}
We will make implicit use of Remark \ref{a000.000} throughout the proof. The
computation of $h^0(C,\mathcal{O}_C(d))$ is well-known.
The technical condition $|B| > dc -\deg (E)$ assures that $h^0(C,\mathcal
{O}_C(d)(-E-B)) =0$. Hence
$\mathcal {C}$ has length $n$ and in fact dimension $k$.
 In the case $E=\emptyset$ the computation of
the minimum distance of $\mathcal {C}^\bot$ is just the planar case of \cite{c},
Proposition 3.1. In the general case notice that $\mathcal {C}$ is obtained
evaluating
a family of homogeneous degree $d$ polynomials on the curve $C$ (the ones
vanishing on the scheme
$E$) at the points of $B$.
Since $C$ is projectively normal (it is a plane smooth curve), the restriction
map
$\rho _d: H^0(\mathbb {P}^2,\mathcal {O}_{\mathbb {P}^2}(d)) \to H^0(C,\mathcal
{O}_C(d))$ is surjective.
Hence the restriction map $\rho _{d,E}: H^0(\mathbb {P}^2,\mathcal {I}_E(d)) \to
H^0(C,\mathcal {O}_C(d)(-E))$ is surjective. Hence a finite subset $S\subseteq
C \setminus E_{red}$
imposes independent condition to $H^0(C,\mathcal {O}_C(d)(-E))$ if and only
if $S$ imposes independent conditions to $H^0(\mathbb {P}^2,\mathcal {I}_E(d))$.
Moreover,
$S$ imposes independent conditions to $H^0(\mathbb {P}^2,\mathcal {I}_E(d))$
if and only if $h^1(\mathbb {P}^2,\mathcal {I}_{E\cup S}(d)) =h^1(\mathbb
{P}^2,\mathcal {I}_E(d))$ (here we use again that $S\cap E_{red}=\emptyset$).
 To get the existence of a non-zero
codeword of $\mC^\perp$ whose support is $S$ (and not only with support
contained in $S$) we need that
the submatrix $M_S$ of the generator matrix of $\mathcal {C}$ obtained by
considering the coloumns associated to the points appearing in $S$ has
the property
that each of its submatrices obtained deleting one coloumn have the same rank of
$M_S$ (each
such coloumn
is associated to some $P\in S$ and we require that the codeword has support
containing  $P$). This is equivalent to the last claim in the statement. 
\end{proof}

\begin{remark}
 Notice that we may drop the assumption $$h^1(\mathbb {P}^2,\mathcal {I}_{E\cup
S}(d)) >h^1(\mathbb {P}^2,\mathcal {I}_{E\cup S'}(d)) \mbox{ for any }
S'\varsubsetneq
S$$ in the statement of Lemma
\ref{e1} if there is no subset $A\subseteq B$ with the properties
$|A|<z$ and $h^1(\mathbb
{P}^2,\mathcal {I}_{E\cup A}(d)) >h^1(\mathbb {P}^2,\mathcal {I}_E(d))$. This is
the case when $z$ is smaller or equal then the Hamming distance of
$\mathcal{C}^{\bot}$.
\end{remark}

\begin{remark}\label{e2}
Take the set-up of the proof of Lemma \ref{e1}. Since both the restriction maps
$\rho
_d$ and
$\rho _{d,E}$ are surjective,  $h^1(\mathbb {P}^2,\mathcal
{I}_{E\cup S}(d)) >h^1(\mathbb {P}^2,\mathcal {I}_E(d))$ is equivalent to 
$h^0(C,\mathcal
{O}_C(d)(-(E\cup S)) > h^0(C,\mathcal {O}_C(d)(-E))-|S|$ or, equivalently
(Riemann-Roch theorem), to
$h^1(C,\mathcal {O}_C(d)(-(E\cup S))
> h^1(C,\mathcal {O}_C(d)(-E))$. In the applications we will usually have $d\le
\deg (C)-3$ and so
$h^1(C,\mathcal {O}_C(d)) >0$.
\end{remark}

\begin{notation}\label{resid}
Let $\F$ be any field and $\Pro^2$ the projective plane over $\F$. Let
$Z\subseteq \mathbb {P}^2$ be any zero-dimensional scheme. Fix a curve
$T\subseteq \mathbb {P}^2$ and set $t:= \deg (T)$ (here we do not assume that
$T$ is
reduced, it may even have multiple components). The residual scheme
$\mbox{Res}_T(Z)$
of $Z$ with respect to the divisor $T$ is defined to be
the closed subscheme of $\mathbb {P}^2$ with $\mathcal {I}_Z:\mathcal {I}_T$ as
its ideal sheaf.
From the general theory of ideal sheaves we have $\mbox{Res}_T(Z) \subseteq T$
and $\deg (Z) = \deg (T\cap Z) +
\mbox{Res}_T(Z)$. If $Z$ is reduced, i.e. if $Z$ is a finite set, then
$\mbox{Res}_T(Z) =
Z\setminus Z\cap T$.
\end{notation}

The following Lemma \ref{u00.01A} a key point in the improvement
of the method of
\cite{c} in the case of plane curves. More precisely, we provide a
cohomological control
over the zero-dimensional scheme $E$ introduced in Section \ref{S1}. The lemma
is in fact
a schematic version of \cite{ep}, Corollaire 2. Parts (a) and (b) also follow in
an arbitrary
projective space from \cite{bgi}, Lemma 34. Parts (b), (c) and  part (d)
are just \cite{ep}, Remarques at page 116.

\begin{lemma}\label{u00.01A}
Let $\F$ be an algebrically closed field and $\Pro^2$ the projective plane over
it. Choose an integer $d>0$ and a zero dimensional scheme 
$Z\subseteq \mathbb {P}^2$. Set $z:= \deg (Z)$.
The following facts hold.
\begin{itemize}
\item[(a)] If $z\le d+1$, then $h^1(\mathbb {P}^2,\mathcal {I}_Z(d))=0$.

\item[(b)] If $d+2 \le z\le 2d+1$, then $h^1(\mathbb {P}^2,\mathcal {I}_Z(d))>0$
if and only if there
exists a line $T_1$ such that $\deg (T_1\cap Z)\ge d+2$.

\item[(c)] If $2d+2\le z \le 3d-1$ and $d\ge 2$, then $h^1(\mathbb
{P}^2,\mathcal {I}_Z(d))>0$ if and only if
either there
exists a line $T_1$ such that $\deg (T_1\cap Z)\ge d+2$, or there exists a conic
$T_2$ such that $\deg (T_2\cap Z) \ge 2d+2$.

\item[(d)] Assume $z=3d$ and $d\ge 3$. Then $h^1(\mathbb {P}^2,\mathcal
{I}_Z(d))>0$ if and only if
either there
exists a line $T_1$ such that $\deg (T_1\cap Z)\ge d+2$, or there exists a conic
$\F$  such that $\deg (T_2\cap Z) \ge 2d+2$, or there exists a plane cubic
$T_3$ such that $Z$ is the complete intersection of $T_3$ and a plane curve of
degree
$d$. In the latter case, if $d\ge 4$ then $T_3$ is unique and we may
find a plane curve $C_d$ with $Z = T_3\cap C_d$.

\item[(e)] Assume $z \le 4d-5$ and $d\ge 4$.  Then $h^1(\mathbb {P}^2,\mathcal
{I}_Z(d))>0$ if and only if
either there
exists a line $T_1$  such that $\deg (T_1\cap Z)\ge d+2$, or there exists a
conic $T_2$  such that $\deg (T_2\cap Z) \ge 2d+2$, or there exist a subscheme
$W\subseteq Z$  with $\deg (W)=3d$ and plane cubic
$T_3$  such that $W$ is  the complete intersection of $T_3$ and a plane curve of
degree
$d$, or there is a plane cubic $C_3$  such that $\deg (C_3\cap Z)\ge 3d+1$.
\end{itemize}
\end{lemma}

\begin{proof}
Since $Z$ is a zero-dimensional scheme, for every $d\in \mathbb {Z}$ and any
closed
subscheme
$W\subseteq Z$ we get $h^1(Z,\mathcal {I}_{W,Z}(d))=0$. Hence
the restriction map $H^0(Z,\mathcal {O}_Z(d)) \to H^0(W,\mathcal {O}_W(d))$ is
surjective.
As a consequence, if $h^1(\mathbb {P}^2,\mathcal {I}_W(d)) >0$, then
$h^1(\mathbb {P}^2,\mathcal {I}_Z(d))>0$. Take any integer $y \in \{1,\dots,
d-1\}$ and any degree $y$ plane curve $D_y$ (we allow $D_y$ to have even
multiple components). Set $W:= D_y\cap Z$. From the exact sequence
\begin{equation}\label{eq0u.01}
0 \to \mathcal {O}_{\mathbb {P}^2}(d-y) \to \mathcal {O}_{\mathbb {P}^2}(d) \to
\mathcal {O}_{D_y}(d)\to 0
\end{equation}
we get $h^0(D_y,\mathcal {O}_{D_y}(d)) = \binom{d+2}{2} -\binom{d-y+2}{2}$ and
that
the restriction map $\rho :H^0(\mathbb {P}^2,\mathcal {O}_{\mathbb {P}^2}(d))
\to H^0(D_y,\mathcal {O}_{D_y}(d))$ is surjective. Hence if $h^0(D_y,\mathcal
{I}_W(d))
> \binom{d+2}{2} -\binom{d-y+2}{2} -\deg (W)$, then $h^1(\mathbb {P}^2,\mathcal
{I}_W(d))>0$
and hence $h^1(\mathbb {P}^2,\mathcal {I}_Z(d))>0$. Since $h^0(D_y,\mathcal
{I}_W(d))\ge 0$,
we will have $h^1(\mathbb {P}^2,\mathcal {I}_W(d))>0$ if $\deg (W)>
\binom{d+2}{2}
-\binom{d-y+2}{2}$.
For $y=1,2$ it is sufficient to assume $\deg (W) \ge yd +2$. For $y=3$ it is
sufficient
to assume $\deg (W) \ge 3d+1$. Now take $y=3$ and $\deg (W)=3d$.
We have $h^0(D_3,\mathcal {I}_W(d))
> \binom{d+2}{2} -\binom{d-3+2}{2} -\deg (W)$ if and only if $h^0(D_3,\mathcal
{I}_W(d)) >0$, i.e.
(by the surjectivity of $\rho$) if and only if there exists a degree $d$ plane
curve
$C_3$ such that
$W = D_3\cap C_3$. Hence in parts (b), (c), (d) and (e) we proved the ``if''
part. In the remaining part of the proof we check part (a) and the ``only if''
part of (b), (c), (d) and (e).
Let $\tau$ be the maximal integer such that $h^1(\mathbb {P}^2,\mathcal
{I}_Z(\tau))>0$ (such a $\tau$ of course exists,
because $h^1(\mathbb {P}^2,\mathcal {I}_Z(t))=0$ for $t\gg 0$, by a famous
theorem of
Serre). By assumption we have $\tau \ge d$. Fix a positive integer
$s\in \{1,2,3,4\}$ and assume $\tau \ge s-3+z/s$. By \cite{ep}, Corollaire 2,
either
$\tau = s-3 +z/s$ and $Z$ is the complete intersection of a degree $s$ plane
curve and a degree $\tau$ curve, or there are $W\subseteq Z$ and an integer
$t\in
\{1,\dots ,\tau -1\}$ such
that $\deg (W) \ge t(\tau -t+3)$ and $W$ is contained in a plane curve of degree
$t$.
\begin{itemize}
 \item[(i)] Parts (a) and (b) are just the plane case of \cite{bgi}, Lemma 34.
\item[(ii)] Now assume $d+2 \le z \le 2d+1$. Since $\tau \ge d$, we have $\tau
\ge d-2$. Take $s=1$. Since $\tau \ge d$, we have $\tau \ge 1+3+z$. Hence we may
apply \cite{ep}, Corollaire 2, and get the existence of $W\subseteq Z$ and a
line $T_1$ such that $\deg (W) \ge
d+2$ and $W \subseteq T_1$.
\item[(iii)] Now assume $2d+2 \le z \le 3d$. Since $\tau \ge d \ge 3-3+z/3$, we
can apply \cite{ep}, Corollaire 2, with the integer $s:= 3$ and get parts (c)
and (d).
\item[(iv)] Finally assume $3d+1 \le z \le 4d-5$. Since $\tau \ge d > 4-3+z/4$,
by \cite{ep}, Corollaire 2, we get part (e). 
\end{itemize}
\end{proof}

In the next section we  study the algebraic geometry of
codes
arising from a Hermitian curve $X$ defined over a 
finite field $\F_{q^2}$ (briefly, of \textbf{Hermitian} codes). In particular,
we will employ the well-known characterization of the tangent lines to $X$ in
order to reveal a
precise cohomological structure in the dual minimum-weight codewords of such 
codes.

\section{The Hermitian curve} \label{hhh}

In this section $q$ denotes a prime power and $\F_{q^2}$ is the finite field
with $q^2$ elements. We denote by $\Pro^2$  
the projective plane over the field $\F_{q^2}$ of projective coordinates
$(x:y:z)$. The \textbf{Hermitian curve} $X \subseteq \Pro^2$ is defined to be
the zero locus of the polynomial $y+y^q=x^{q+1}$ (as an
affine equation). This
curve is clearly defined over $\F_{q^2}$. Its function field is studied in
\cite{Sti}, Example 6.3.6. From a geometric point of view, $X$ is
known to be a smooth curve of genus $g(X)=q(q-1)/2$ given by the
genus-degree
formula. The projective geometry of tangent lines to $X$ is completely known and
summarized in the following two results.

\begin{lemma}
\label{intrette}
Let $X$ be the Hermitian curve. Every line $L$ of $\Pro^2$ either intersects $X$
in $q+1$ distinct $\F_{q^2}$-rational 
points, or $L$ is tangent to $X$ at a point $P$ (with contact order $q+1$). In
the latter case $L$ does not intersect $X$ in any other 
point different from $P$.
\end{lemma}
\begin{proof}
See \cite{h}, part (i) of Lemma 7.3.2, at page 247. 
\end{proof}

\begin{lemma}\label{u5.0}
Fix an integer $e\in \{2,\dots ,q+1\}$ and $P\in 
 X(\mathbb {F}_{q^2})$. Let $E\subseteq X$ be the divisor $eP$, seen as a closed
degree $e$ subscheme of $\mathbb {P}^2$. Let $T\subseteq \mathbb {P}^2$ any
effective divisor
 (i.e. a plane curve, possibly  with multiple components)
 of degree smaller or equal than $e-1$ and such that the contact order of $X$
and $T$ at $P$ is at least $e$. Then $L_{X,P}\subseteq T$, i.e. $L_{X,P}$
is one of the components of $T$.
\end{lemma}

\begin{proof}
Since $L_{X,P}$ has order of contact $q+1\ge e$ with $X$ at $P$, we have
$E\subseteq L_{X,P}$.
Since $\deg (E)>\deg (T)$ and $E\subseteq T\cap L_{X,P}$, Bezout theorem implies
$L_{X,P}\subseteq T$. 
\end{proof}

The Hermitian curve carries $|X(\F_{q^2})|=q^3+1$ rational points (\cite{Sti},
p. 250, or \cite{h}). It follows that $X$ is a maximal curve, in the sense
of the Hasse-Weil bound (see \cite{Ste}, Chapter 6). 

\begin{notation}
 For any point $P\in X(\mathbb {F}_{q^2})$ we denote by $L_{X,P} \subseteq
\mathbb {P}^2$ the
tangent line to $X$ at $P$.
Clearly, $L_{X,P}$ is a line defined over $\mathbb {F}_{q^2}$. 
\end{notation}

\begin{lemma}\label{u5.00}
Fix integers $d$, $s$ such that  $d\ge s\ge 1$. Choose $s$ distinct points
$P_1,...,P_s\in 
 X(\mathbb {F}_{q^2})$ and $s$ integers $b_1,...,b_s$ such that  $0\le b_i\le
d+2-i$ and $b_i\le q+1$
 for any $i \in \{1,...,s \}$.
 Let $E:= \sum _{i=1}^{s} b_iP_i$ be a degree $b_1+\dots +b_s$ effective
divisor
 of $X$, seen also as a degree $b_1+\dots +b_s$ zero-dimensional
 subscheme of $\mathbb {P}^2$. Then we have $h^1(\mathbb {P}^2,\mathcal
{I}_E(d)) =0$.
\end{lemma}

\begin{proof}
For any integer $j\in \{1,\dots ,s\}$ set $E[j]:= \sum _{i=j}^{s} b_iP_i$ and
$E_i:= b_iP_i$. Hence $E[1]=E$ and $E[i] = \sqcup _{i \le j \le s} E_j$.
We can see each $E[i]$ as a  degree $b_i+\dots +b_s$ zero-dimensional
 subscheme of $\mathbb {P}^2$.
 Since $L_{X,P_i}$ has order of contact $q+1$ with $X$ at $P_i$ and $b_i\le
q+1$,
 we have $E_i\subseteq L_{X,P_i}$. Hence $E[i+1] = E[i]\setminus E[i]\cap
L_{X,P_i}$
 for any $i= 1,\dots ,s-1$. See $E[i]$ and $E[i+1]$ as zero-dimensional
subschemes of $\mathbb {P}^2$ and $L_{X,P_i}$ as a degree $1$ curve of $\mathbb
{P}^2$. Then for any
 $t\in \mathbb {Z}$ and any $i\in \{1,...,s \}$ we get the following exact
sequence
 of coherent sheaves on $\mathbb {P}^2$:
 \begin{equation}\label{equ1}
 0 \to \mathcal {I}_{E[i+1]}(t-1) \to \mathcal {I}_{E[i]}(t) \to \mathcal
{I}_{E_i,L_{X,P_i}}(t) \to 0
 \end{equation}
 in which we see $E_i$ as a degree $b_i$ divisors of $L_{X,P_i} \cong \mathbb
{P}^1$. Hence $h^1(L_{X,P_i},\mathcal
{I}_{E_i,L_{X,P_i}}(t))=0$ for any
$t\ge b_i+1$. Taking $t=d$ and $i=1$ in (\ref{equ1}) we get
 $h^1(\mathbb {P}^2,\mathcal {I}_E(d)) \le h^1(\mathbb {P}^2,\mathcal
{I}_{E[2]}(d-1))$.
 If $s=1$ then we are done, because $E[2]=\emptyset$ in this case. In the
general
case
 we use induction on $s$. Notice that we may apply the inductive
 assumption to $E[2]$ with respect to the integer $d':= d-1$. Hence the
inductive
 assumption gives $h^1(\mathbb {P}^2,\mathcal {I}_{E[2]}(d-1))=0$. Conclude by
using the long
 cohomology exact sequence of (\ref{equ1}) in the case $t=d$ and $i=1$. 
\end{proof}

\begin{remark}\label{goppa}
Fix an integer $s\ge 2$ and take $s$ distinct
points $P_1,..., P_s\in X(\mathbb {F}_{q^2})$. Choose integers $a_1,...,a_s$ and
set $E:=\sum_{i=1}^s a_iP_i$, 
both viewed as a divisor on $X$ and as a zero-dimensional subscheme $E
\subseteq \Pro^2$. For any $\F_{q^2}$-rational point $P\in X(\mathbb {F}_{q^2})$
we have an isomorphism of sheaves $\mathcal {O}_X((q+1)P) \cong \mathcal
{O}_X(1)$. It follows that there exists a rational function $f_P$ such that
$(f_P) = (q+1)P-(q+1)Q$. Set $G:=d(q+1)P_1-E$ and denote by $\mathcal{L}(G)$ the
Riemann-Roch space associated to the divisor $G$. The codes 
obtained evaluating at the points of $B:= X(\F_{q^2})\setminus \{ P_1,...,P_s\}$
the vector spaces
$$H^0(X,\Ol_X(d)(-E)) \ \ \ \ \ \ \mbox{and} \ \ \ \ \ \ \ \mathcal{L}(G)$$
are in fact strongly isometric (see Definition \ref{siso}). 
Notice that codes obtained evaluating a Riemann-Roch space $\mathcal{L}(G)$ at
the rational points of a curve avoiding the support of $G$
are the famous \textbf{Goppa} codes (see \cite{Ste} for further details). We
will denote by $\mC(B,d,-E)$ the code obtained evaluating the vector space
$H^0(X,\Ol_X(d)(-E))$ on the set $B$ defined above and by $\mC^\perp(B,d,-E)$
its dual code. This remark shows that the class of $\mC(B,d,-E)$ codes, here
studied, includes many codes of classical interest in geometric coding theory.
\end{remark}

In the following result the geometry of tangent lines to $X$ is explicitly
involved in proving that we may restrict, in studying 
Hermitian codes, to a very particular subclass of them.

\begin{lemma}\label{u4}
Fix integers $d>0$ and $s\ge 1$. Choose $s$ integers $a_1,...,a_s \in \{1,\dots
,q+1\}$
with $a_1\le \cdots \le a_s$ and denote by $r$ be the maximal integer $i\le s$
with the property 
$a_i\le d-s+i$. Set $d':= d-s+r$ and assume $d'>0$. Set $a'_i:= a_i$ for any
$i\le r$. Fix
$s$ distinct points $P_1,..., P_s\in X(\mathbb {F}_{q^2})$ and set $$E:= \sum
_{i=1}^{s}
a_iP_i, \ \ \ \ \ E':= \sum _{i=1}^{r}a_iP_i.$$ Fix any
$B\subseteq X(\mathbb {F}_{q^2})\setminus \{P_1,..., P_s\}$. Then the codes
$\mathcal {C}(B,d,-E)$ and $\mathcal {C}(B,d',-E')$
are strongly isometric. In particular, their dual codes are the strongly
isometric.
\end{lemma}

\begin{proof}
If $r=s$ then $d' =d$, $E' =E$ and so
there is noting to prove. Assume $r<s$, i.e. $a_r>d$.
Take any
$f\in H^0(X,\mathcal {O}_X(d)(-E))$, which is a  degree $d$
homogeneous
polynomial vanishing on the
zero-dimensional
scheme $E\subseteq \mathbb {P}^2$. Fix any $i\in \{r+1,..., s\}$.
 Lemma \ref{u5.0} gives that $f$ is divided by the equation of the
tangent line $L_{X,P_i}$ to $X$ at $P_i$. The division by the equations of the
tangent lines
$L_{X,P_i}$, $r+1 \le i \le s$, gives an isomorphism of vector spaces
$$H^0(X,\Ol_X(d)(-E)) \to H^0(X,\Ol_X(d')(-E')).$$
Since a tangent line to $X$ at a rational point $P$ does not intersect $X$ at
any other rational point, the codes
$\mathcal{C}(B,d,-E)$ and $\mathcal{C}(B,d',-E')$ are in fact strongly
isometric. Conclude by Remark \ref{siso} 
\end{proof}

 We notice that Lemma \ref{u00.01} works over
an algebrically closed field.  The following result improves Lemma
\ref{u00.01A} for the case of the Hermitian curve over finite fields.

\begin{lemma}\label{u00.01}
Let $\F_{q^2}$ be the finite field with $q^2$ elements ($q$ a prime power) and
denote by $\Pro^2$ the projective plane over the field $\F_{q^2}$. Let
$X\subseteq \Pro^2$ be the Hermitian curve. Choose an integer $d>0$ and a
zero-dimensional scheme $Z \subseteq X(\F_{q^2})$ of degree $z>0$. The following
facts hold.

\begin{itemize}
\item[(a)] If $z\le d+1$, then $h^1(\mathbb {P}^2,\mathcal {I}_Z(d))=0$.

\item[(b)] If $d+2 \le z\le 2d+1$, then $h^1(\mathbb {P}^2,\mathcal {I}_Z(d))>0$
if and only if there
exists a line $T_1$ such that $\deg (T_1\cap Z)\ge d+2$.

\item[(c)] If $2d+2\le z \le 3d-1$ and $d\ge 2$, then $h^1(\mathbb
{P}^2,\mathcal {I}_Z(d))>0$ if and only if
either there
exists a line $T_1$ defined over $\F_{q^2}$ such that $\deg (T_1\cap Z)\ge d+2$,
or there exists a conic $T_2$ defined over $\F_{q^2}$ such that $\deg (T_2\cap
Z) \ge 2d+2$.

\item[(d)] Assume $z=3d$ and $d\ge 3$. Then $h^1(\mathbb {P}^2,\mathcal
{I}_Z(d))>0$ if and only if
either there
exists a line $T_1$ defined over $\F_{q^2}$ such that $\deg (T_1\cap Z)\ge d+2$,
or there is a conic $T_2$ defined over $\F_{q^2}$ such that $\deg (T_2\cap Z)
\ge 2d+2$, or there exists a plane cubic
$T_3$ such that $Z$ is the complete intersection of $T_3$ and a plane curve of
degree
$d$. In the latter case, if $d\ge 4$ then $T_3$ is unique and defined over
$\F_{q^2}$ and we may
find a plane curve $C_d$ defined over $\F_{q^2}$ and with $Z = T_3\cap C_d$.

\item[(e)] Assume $z \le 4d-5$ and $d\ge 4$.  Then $h^1(\mathbb {P}^2,\mathcal
{I}_Z(d))>0$ if and only if
either there
exists a line $T_1$ defined over $\F_{q^2}$ such that $\deg (T_1\cap Z)\ge d+2$,
or there exists a conic $T_2$ defined over $\F_{q^2}$ such that $\deg (T_2\cap
Z) \ge 2d+2$, or there exist $W\subseteq Z$ defined over $\F_{q^2}$ with $\deg
(W)=3d$ and plane cubic
$T_3$ defined over $\F_{q^2}$ such that $W$ is  the complete intersection of
$T_3$ and a plane curve of degree
$d$, or there is a plane cubic $C_3$ defined over $\F_{q^2}$ such that $\deg
(C_3\cap Z)\ge 3d+1$.
\end{itemize}

\end{lemma}

\begin{proof}
  Let $Z_{\overline{\F_{q^2}}}$ denote the scheme $Z$, but seen over
$\overline{\F_{q^2}}$. Assume to be in one of the cases
(a), (b), (c), (d) or (e) of the statement.
By Remark \ref{a000.000} the cohomological non-vanishing condition is the same
over $\F_{q^2}$
or over $\overline{\F_{q^2}}$. Hence (a) holds for $Z$ over $\F_{q^2}$, while in
the other case we need
to inquire whether the curves of low degree claimed in the statement are defined
over $\F_{q^2}$, or not. First assume $z = 3d$ and that
$Z_{\overline{\F_{q^2}}}$ is the complete
intersection of a plane cubic and a degree $d$ curve. By Remark \ref{a000.000}
the homogeneous ideal of $Z$ in $\mathbb {P}^2$ is generated by forms defined
over $\F_{q^2}$. Hence $Z$ is the complete
intersection of a plane cubic defined over $\F_{q^2}$ and a degree $d$ curve
defined over $\F_{q^2}$. The
other cases are in general more complicated, but in the applications to the
Hermitian curve
we know more about $Z$: not only $Z$ is defined over $\F_{q^2}$, but each
connected component
of it is defined over $\F_{q^2}$ and hence $Z_{red} \subseteq \mathbb
{P}^2_{\F_{q^2}}$; moreover, we
also know that the lines $D$ with $\deg (D\cap Z) \ge 2$ are defined over
$\F_{q^2}$. Hence it is sufficient to notice the following facts.
\begin{itemize}
 \item[(1)] A line of $\mathbb {P}^2_{\overline{\F_{q^2}}}$ containing two
points of
$\mathbb {P}^2$
is defined over $\F_{q^2}$.
\item[(2)] A conic of $\mathbb {P}^2_{\overline{\F_{q^2}}}$ containing $5$
points of
$\mathbb {P}^2$, no
$4$ of them on a line, is defined over $\F_{q^2}$.
\item[(3)] Let $C$ be a plane cubic defined over $\overline{\F_{q^2}}$ and
containing
a set $S\subseteq \mathbb {P}^2$ of at least 12 points. Assume that no $5$ of
the points of $S$ are contained
in a line and no $8$ of the points of $S$ are contained in a conic. Then $C$ is
defined over $\F_{q^2}$.
\end{itemize} 
\end{proof}

\section{Hermitian codes} \label{hhh2}

In this last section we apply our results and describe the dual minimum distance
of many $s$-point codes on the 
Hermitian curve (with $s \ge 2$). For $s \ge 3$ this parameter was unknown,
except in some particular cases covered by \cite{c}. Until the end of the paper
we work over the finite field $\F_{q^2}$ with $q^2$ elements ($q$ fixed) and we
will denote by $X \subseteq \Pro^2$ the Hermitian curve defined in Section
\ref{hhh}. The following Theorem \ref{u5} and Theorem \ref{m1} describe
Hermitian three-point codes, while Theorem \ref{u0.1} and Theorem \ref{m3} deal
with the more complicated case of $s$-point codes with $s \ge 2$ arbitrary.

\begin{theorem}\label{u5}
Fix any three distinct points $P_1,P_2,P_3\in 
 X(\mathbb {F}_{q^2})$ and assume $P_1$, $P_2$ and $P_3$ to be not collinear.
 Set $B:=  X(\mathbb {F}_{q^2})\setminus \{P_1,P_2,P_3\}$. Fix
an integer $d \ge 5$ such that $1\le d\le q-1$ and integers $a_1,a_2,a_3 \in
\{1,\dots ,d\}$
such that $a_1+a_2+a_3 \le 3d-5$ and $a_i=d$ for at most one index $i \in \{
1,2,3\}$. Set $E:= a_1P_1+a_2P_2+a_3P_3$. 
 Let $\mathcal {C}:= \mathcal {C}(B,d,-E)$ be the code obtained evaluating the
vector space $H^0(X,\Ol_X(d)(-E))$ on the set $B$. Then $\mathcal {C}$ is a code
of length $n:=|B|=q^3-2$ and dimension $k:= \binom{d+2}{2} -a_1-a_2-a_3$. For
any $i \in \{ 1,2,3 \}$ let $L_i$ denote the line spanned
 by $P_j$ and $P_h$ with $\{i,j,h\} = \{1,2,3\}$. Then $\mathcal {C}^{\bot}$ has
minimum distance $d$ and 
 its minimum-weight codewords are exactly the ones whose support is formed by
$d$ points of $B\cap L_i$
 for some $i \in \{ 1,2,3\}$.
 \end{theorem}
\begin{proof}

The length of $\mathcal{C}$ is obviously $n = |B| = q^3-2$. Since $d\le q < \deg
(X)$, we have $h^0(X,\mathcal {O}_X(d)) =\binom{d+2}{2}$. If, say, $a_1\ge
a_2\ge a_3$, the assumptions $a_1\le d$ and $a_1+a_2+a_3   \le 3d-1$
give $a_i\le d+2-i$ for all $i$. Hence Lemma \ref{u5.00} implies $h^1(\mathbb
{P}^2,\mathcal {I}_E(d)) = 0$ and so $ h^0(X,\mathcal {O}_X(d)(-E)) =
\binom{d+2}{2} -a_1-a_2-a_3 =k$. Since $|B|
> d\cdot \deg (X)$, there is not a non-zero element
of $H^0(X,\mathcal {O}_X(d))$ vanishes at all the points of $B$. Hence $\mathcal
{C}$ has dimension $k$.
By Lemma \ref{e1} it is sufficient to prove the following two facts. 
\begin{itemize}
 \item[(a)] $h^1(\mathbb {P}^2,\mathcal {I}_{E\cup A}(d)) =0$ for all
$A\subseteq B$ such that $|A|\le d-1$.
\item[(b)] For any
$S\subseteq B$ such that $|S| =d$ we have
$h^1(\mathbb {P}^2,\mathcal {I}_{E\cup S}(d)) >0$ if and only if $S\subseteq
L_i$ for some $i \in \{ 1,2,3\}$.
\end{itemize}
Each line $L_i$ contains $q-1$ points of $B$, while $\deg (E\cap L_i)=2$.
Hence for any $S\subseteq L_i\cap B$ with $|S| = d$ we have $h^1(\mathbb
{P}^2,\mathcal {I}_{E\cup S}(d)) >0$ (see Lemma
\ref{u00.01}). Let $E_i:=a_iP_i$, view as a divisor on $X$. We have $E
=E_1\sqcup E_2\sqcup E_3$. Fix a set $S\subseteq B$ such that $|S|\le d$
and assume $h^1(\mathbb {P}^2,\mathcal {I}_{E\cup S}(d)) >0$. We have $S\cap
\{P_1,P_2,P_3\} =\emptyset$ and $\deg (E\cup S) = a_1+a_2+a_3+|S|$.
Since $a_1+a_2+a_3 +|S| \le 4d-5$, we may apply
Lemma \ref{u00.01} to the scheme $E\cup S$. Let $T\subseteq \mathbb {P}^2$ be
the curve arising from the statement of the lemma. Set $x:= \deg (T) \in
\{1,2,3\}$ and $e_i:= \deg (T\cap E_i)$ for $i \in \{1,2,3\}$. We have $0 \le
e_i\le a_i$. If $e_i\ge x+1$ then
Lemma \ref{u5.0} gives $L_{X,P_i}\subseteq T$. Assume $e_i\le x$ for all $i \in
\{ 1,2,3\}$. 
For $x=2$ we get $\deg (T\cap (E\cup S)) \le d+6 \le 2d+1$.
For $x=3$ we get $\deg (T\cap (E\cup S)) \le d+9 \le 3d-1$. Finally, for $x=1$
we may have
$e_i>0$ only for at most two indices, say $i=1,2$. Since $|S|\le d$, we get
$|S|+e_1+e_2\ge d+2$ and $|S|+e_1+e_2=d+2$ if and only if $T =L_3$, $S\subseteq
L_3\cap B$ and $|S|=d$.
Now assume that $T$ contains one of the lines $L_{X,P_i}$, say $L_{X,P_1}$. Let
$T'$ be the curve whose
equation is obtained dividing an equation of $T$ by an equation of $L_{X,P_1}$.
We have
$\deg (T') =x-1$, $T'+L_{X,P_1} = T$ (as divisors of $\mathbb {P}^2$) and $T =
L_{X,P_1}\cup T'$
(as sets). Since $L_{X,P_1}\cap B=\emptyset$, we have $T\cap S
= T'\cap S$ and $\deg (T\cap (E\cup S)) = \deg (T'\cap (E_2\cup E_3\cup S))$.
\begin{itemize}

 \item[(i)] If $x=1$, we
get $T\cap S=\emptyset$ and $\deg (T\cap E)
= a_1 \le d$, a contradiction.

\item[(ii)] Assume $x=2$.  The curve $T'$ must be a line such that $\deg (T'\cap
(E_2\cup E_3 \cup S))
\ge 2d+2-a_1$. If either $T' = L_{X,P_2}$, or $T' = L_{X,P_3}$, we
get $T'\cap S = \emptyset$ and $\deg (T'\cap (E_2\cup E_3\cup S)) \le \max
\{e_2,e_3\} \le d$, a contradiction. If neither $T' = L_{X,P_2}$, nor $T' =
L_{X,P_3}$, then $\deg (T'\cap E_2)\le 1$,
$\deg (T\cap E_3)\le 1$ and $\deg (T'\cap (E_2\cup E_3)) =2$ if and only if $T'
=L_1$.
Since $|S| \le d$ we deduce $\deg (T\cap (E\cup S))  \le a_1+2+|S|$. Moreover,
the equality holds
if and only if $T' = L_1$ and $S\subseteq L_1$. Since
$\deg (T\cap (E\cup S)) \ge 2d+2$ by assumption, $|S|=d$ and $S\subseteq L_1$,
as claimed.

\item[(iii)]Now
assume $x=3$. We get $\deg (T'\cap (E_2\cup E_3\cup S)) \ge 3d-a_1$ and $T'$ is
a conic.
If neither
$L_{X,P_2}$, nor $L_{X,P_3}$, is a component of $T$
then Lemma \ref{u5.0} gives $e_2\le 2$ and $e_3\le 2$ and so $|T'\cap S| \ge
3d-4-a_1
\ge 2d-4 > d$. If, say, $T'$ contains
$L_{X,P_2}$ and $T''$ is the line with $T' = T''+L_{X,P_2}$, then we get
$|(S\cup E_3)\cap T''| \ge 3d-a_1-a_2$. Since
$a_1+a_2 \le 2d-1$ we deduce $\deg (T''\cap (E_3\cup S)) \ge d+1$. Since $\deg
(T''\cap E_3)\le 1$,
we get
$a_1+a_2=2d-1$, say $a_1=d$ , $a_2=d-1$ and that $S$ is formed by $d$ points on
a line
$T''$ through $P_3$. If either $T''= L_1$ or $T''= L_3$, then we are done. In
any case it is sufficient to prove that
the case $x=3$ of Lemma \ref{u00.01} does not apply, i.e. that $E_1\cup E_2\cup
\{P_3\}\cup S$ is not
the complete intersection of $T = L_{X,P_1}\cup L_{X,P_2}\cup T''$ and a degree
$d$ curve, say $C_d$.
Since $a_2=d-1$, $E_2$ is not the complete intersection of $L_{X,P_2}$ and
$C_d$, while $L_{X,P_2}\cap (\{P_3\}\cup S)
= \emptyset$, a contradiction.
\end{itemize}
\end{proof}

\begin{remark}
 Let us compare the results of Theorem \ref{u5} with the classcal Goppa bound
for codes obtained from curves. Choose $B$, $d$, $P_1$, $P_2$, $P_3$, $a_1$,
$a_2$, $a_3$  as in the theorem. Denote by $P_\infty:=(0:1:0)$ the point at
infinity of
the Hermitian curve $X$. As in Remark \ref{goppa}, we have
$$H^0(X,\Ol_X(d)) = \mathcal{L}(d(q+1)P_\infty).$$
As a consequence, $\mathcal{C}:=\mathcal {C}(B,d,-E)$ is the code obtained
evaluating the Riemann-Roch space $\mathcal{L}(G)$ at the points of $B$, where
$G:=d(q+1)P_\infty - a_1P_1-a_2P_2-a_3P_3$.
The dual minimum distance of $\mathcal{C}$ is lower-bounded by $\deg(G)-(2g-2)$
(see \cite{Sti}, Theorem 2.2.7). Since $d \le q-1$, we have 
$$\deg(G)-(2g-2)=d(q+1)-a_1-a_2-a_3-q(q-1)+2 \le 
q+1-a_1-a_2-a_3.$$
As a consequence, if $d>q+1-a_1-a_2-a_3$ then the dual minimum distance of
$\mathcal{C}$ is higher than the designed distance $\deg(G)-(2g-2)$.

\end{remark}

\begin{remark}\label{m2}
Take the set-up of Theorem \ref{u5} and assume $d\ge 6$, $a_1=a_2=d$ and $1 \le
a_3\le d-5$. Take any line $L$ through $P_3$ with $L\ne L_{X,P_3}$ and
any $S\subseteq L\cap B$ such that $|S|=d$. Then $S$ is the support
of a codeword of $\mathcal {C}^{\bot}$ of weight $d$. Indeed, Lemma \ref{u5.00}
gives $h^1(\mathbb {P}^2,\mathcal {I}_E(d))=0$. Hence it is enough to prove that
$h^1(\mathbb {P}^2,\mathcal {I}_{E_1\cup E_2\cup \{P_3\}\cup S}(d)) >0$. Since
$\deg (E_1\cup E_2\cup \{P_3\} \cup S) = 3d+1$ and $E_1\cup E_2\cup \{P_3\}\cup
S$ is contained
into the degree $3$ curve $L_{X,P_1}\cup L_{X,P_2}\cup L$, we can simply apply
the
``~if~'' part of Lemma \ref{u00.01}, part (e).
\end{remark}

In the following result we modify the set $B$ in the statement of Theorem
\ref{u5} in order to obtain (under certain assumptions) $\mC^\perp$ codes of
improved parameters. To be precise, we will exhibit codes whose dual minimum
distance is $d+1$ instead of $d$.

\begin{theorem}\label{m1}

Fix any three distinct points $P_1,P_2,P_3\in 
 X(\mathbb {F}_{q^2})$ and assume $P_1$, $P_2$ and $P_3$ to be not collinear.
Fix
an integer $d \ge 5$ such that $1\le d\le q-1$ and integers $a_1,a_2,a_3 \in
\{1,\dots ,d\}$
such that $a_1+a_2+a_3 \le 3d-5$ and $a_1+a_2 \le 2d-2$. Set $E:=
a_1P_1+a_2P_2+a_3P_3$ and $B':= X(\mathbb {F}_{q^2})\setminus (X(\mathbb
{F}_{q^2})\cap (L_1\cup L_2\cup L_3))$, where $L_1, L_2,L_3$ are the three lines
spanned by $P_1,P_2,P_3$. Let $\mathcal{C}:=\mathcal{C}(B',d,-E)$ be the code
obtained evaluating the vector space $H^0(X,\Ol_X(d)(-E))$ on the set $B'$. Then
$\mathcal{C}$ has length $n=q^3-3q+1$ and dimension
$k:=\binom{d+2}{2}-a_1-a_2-a_3$.
Let $\mathcal {S}$ denote be the set of all lines in $\Pro^2$ defined over
$\mathbb {F}_{q^2}$ through one of the points $P_1,P_2,P_3$, but
 different from $L_1,L_2,L_3$ and from the tangent lines $L_{X,P_i}$, $i=1,2,3$,
to $X$ at
 $P_i$. Let $\mathcal {S}(d+1)$ denote the set
 of all the subsets $S\subseteq B$ such that $|S|=d+1$ and contained in some
line $L\in\mathcal {S}$. We have $|\mathcal {S}(d+1)| = 3(q^2-2)\binom{q}{d+1}$.
The minimum distance of $\mathcal{C}^{\bot }$
 is $d+1$ and for each $S\in  \mathcal {S}(d+1)$ there exists a minimum-weight
codeword (unique up to a scalar multiplication) with $S$ as its support.
 If $d\ge 6$ and $a_1+a_2+a_3\le 3d-6$ then all the minimum-weight codewords of
$\mathcal{C}^{\bot }$ 
 arise in this way from a unique $S\in \mathcal {S}(d+1)$.
\end{theorem}

\begin{proof}
 
Since $|L_i\cap X(\mathbb {F}_{q^2})|=q+1$
for any $i \in \{1,2,3\}$, we have $|X(\mathbb {F}_{q^2})\cap (L_1\cup L_2\cup
L_3)| =3q$.
Hence $|B'| = q^3-3q+1$. Lemma \ref{u5.00} gives $h^0(X,\mathcal {O}_X(d)(-E))
= \binom{d+2}{2} -a_1-a_2-a_3$. Since $E\subseteq X$ and $|B'|+\deg (E)
> d\cdot \deg (X)$, there is not non-zero element of $H^0(X,\mathcal
{O}_X(d)(-E))$ vanishing at any point
of $B'$. Hence $\mC$ has dimension $k = \binom{d+2}{2} -a_1-a_2-a_3$.
Theorem \ref{u5} implies that $\mathcal {C}^{\bot}$ has minimum distance at
least $d+1$. The set of all the lines
through any $P_i$ has cardinality $q^2+1$. One of these lines is the tangent
line $L_{X,P_i}$
and two of these lines are in $\{ L_1,L_2,L_3\}$. Hence $|\mathcal {S}|
=3(q^2-2)$. Since $|B'\cap L| = q$ for any $L\in \mathcal {S}$, we have
$|\mathcal {S}(d+1)| = 3(q^2-2)\binom{q}{d+1}$. Fix any $S\in \mathcal
{S}(d+1)$.
We have $h^1(\mathbb {P}^2,\mathcal {I}_E(d))=0$ (by Lemma \ref{u5.00}). Parts
(a) and (b) of Lemma \ref{u00.01}
give $h^1(\mathbb {P}^2,\mathcal {I}_{E\cup S}(d)) >0$. Hence Lemma \ref{e1}
tells us that
$S$ is the support of a unique (up to a non-zero scalar) minimum-weight codeword
of $\mathcal{C}^\perp$.
Now assume $d\ge 6$ and
$a_1+a_2+a_3\le 3d-6$. Look at the proof of Theorem \ref{u5}. Fix any
$S\subseteq B'$ such that $|S|=d+1$. Since $a_1+a_2+a_3+|S|\le 4d-5$, we may
apply Lemma
\ref{u00.01}. Let $T$ a curve arising from the Lemma and set $x:= \deg (T)\in
\{1,2,3\}$ and $e_i:= \deg (E_i\cap T)$.
First assume $L_{X,P_i}\nsubseteq T$ for any $i \in \{ 1,2,3\}$.

\begin{itemize}
 \item[(i)] If $x=1$ then we get $\deg (E\cap T)\le 1$
and $\deg (E\cap T)>0$ if and only if $P_i\in T$. Since $S\cap T\ne \emptyset$,
we get
$T\in \mathcal {S}$ and $S\in \mathcal {S}(d+1)$.
\item[(ii)] If $x=2$, then we have $e_1+e_2+e_3 \le 6$.
Since $\deg (T\cap (E\cup S)) \ge 2d+2$, we get $d+7\ge 2d+2$, a contradiction.
\item[(iii)] Now assume
$x=3$. Since $\deg (T\cap (E\cup S)) \ge 3d$ and $e_i\le 3$ for any $i \in \{
1,2,3\}$, we get $d\le 5$, a contradiction.
\end{itemize}
 From now on we assume $L_{X,P_i}\subseteq T$ for some $i \in \{1,2,3\}$ and
write $T = L_{X,P_i}+T'$. In the case
$x=1$ we obviously get a contradiction. In the case $x=2$ we have $\deg ((S\cup
E_j\cup E_h)\cap T') \le |
S|+1$ (for $\{i,j,h\} = \{1,2,3\}$), because $L_i\cap B' = \emptyset$. This is a
contradiction.  Now assume $x=3$ and that no $L_{X,P_j}$, $j \ne i$, is
contained in $T'$. It follows
$e_j+e_h \le 4$ (for $\{i,j,h\} = \{1,2,3\}$). We deduce $a_i+4 +d+1 \ge 3d$, a
contradiction. Now assume
$L_{X,P_j} \subseteq T'$, say $T' = R+L_{X,P_j}$. Conclude by using the last
part of the proof of Theorem \ref{u5} by setting $|S|=d+1$ and $a_1+a_2\le
2d-2$.
\end{proof}

The following two results are the analogous of Theorem \ref{u5} and Theorem
\ref{m1} for the more complicated case of Hermitian $s$-point codes, with $s \ge
2$ arbitrary.

\begin{remark}
 In \cite{m4} G. L. Matthews computed the Weierstrass semigroup of any $s$
collinear points in
$X(\mathbb {F}_{q^2})$. In particular, the dimensions of the vector spaces of
the form $H^0(X,\mathcal {O}_X(t)(-\sum _{i=1}^{s}a_iP_i))$, with $s \ge 1$,
$P_1,...,P_s \in X(\F_{q^2})$ and $t,a_1,...,a_s\ge 0$ arbitrary, are known.
\end{remark}

\begin{lemma}\label{c1}

Fix integers integers $s$ and $d$ such that and $2 \le s \le d -1\le  q-2$.
Choose integers $a_1,...,a_s$ and fix $s$ distinct collinear points $P_1,...,P_s
\in  X(\mathbb {F}_{q^2})$. Denote by $R$ the line containing the $s \ge 2$
points $P_1,...,P_s$ and take $B:=  X(\mathbb {F}_{q^2})\setminus
\{P_1,...,P_s\}$. Set $E:= \sum _{i=1}^{s} a_iP_i$.
Let $S\subseteq B$ be any subset. For any integer $t$ we
have an exact sequence of coherent sheaves:
\begin{equation}\label{eqc1}
0 \to \mathcal {I}_{E'\cup (S\setminus R\cap S)}(t-1) \to \mathcal {I}_{E\cup
S}(t) \to \mathcal {I}_{\{P_1,\dots ,P_s\}\cup (S\cap R)}(t) \to 0
\end{equation}
in which $\{P_1,\dots ,P_s\} \cup (S\cap R)$ is a set of $s+|S\cap R|$ points of
$R$.
For each integer $i\ge 0$ we have
\begin{equation}\label{eqc2}
h^i(\mathbb {P}^2,\mathcal {I}_{E\cup S}(t))
\le h^i(\mathbb {P}^2,\mathcal {I}_{E'\cup (S\setminus S\cap R}(t-1)) +
h^i(R,\mathcal {I}_{\{P_1,\dots ,P_s\}\cup (S\cap
R),R}(t))
\end{equation}
If $t\ge |S\cap R|+s-1$, then $h^1(\mathbb {P}^2,\mathcal {I}_{E\cup S}(t))
\le h^1(\mathbb {P}^2,\mathcal {I}_{E'\cup (S\setminus S\cap R)}(t-1))$.
\end{lemma}

\begin{proof}
For any closed subscheme $Z\subseteq \mathbb {P}^2$ the zero-dimensional scheme
$\mbox{Res}_L(Z)$ is the closed subscheme of $\mathbb {P}^2$ with $\mathcal
{I}_Z:\mathcal {I}_R$ as its ideal sheaf (see Notations \ref{resid}). For
any finite set $A\subseteq \mathbb {P}^2$ we have $\mbox{Res}_R(S) =S\setminus
S\cap R$. Since $R$ is a degree $1$ divisor of $\mathbb {P}^2$ we have
the residual sequence
\begin{equation}\label{eqc3}
0 \to \mathcal {I}_{\mbox{Res}_R(Z)}(t-1) \to \mathcal {I}_{Z}(t) \to \mathcal
{I}_{Z\cap R,R}(t) \to 0
\end{equation}
Take $Z:= E\cup S$ with $S\subseteq B$. We have $\mbox{Res}_R(S) = S\setminus
S\cap R$. Since $R$ is not tangent to $C$ and $P_i\in R$ for
any $i \in \{ 1,...,s\}$, we have
$\mbox{Res}_R(E) = E'$. Hence $\mbox{Res}_R(E\cup S) = E'\cup (S\setminus S\cap
R)$. Since $P_i\in R$ for any $i$ and $R$ is
transversal to $C$, we have $R\cap E = \{P_1,\dots ,P_s\}$ and hence $R\cap
(E\cup S) = \{P_1,\dots ,P_s\} \cup (S\cap R)$.
Applying (\ref{eqc3}) we get (\ref{eqc1}). The cohomology exact sequence induced
by (\ref{eqc1}) gives (\ref{eqc2}).
Since $R\cong \mathbb {P}^1$ and  $\deg ( \{P_1,\dots ,P_s\} \cup (S\cap R))
=s+|S\cap R|$,
we have $h^1(R,\mathcal {I}_{\{P_1,\dots ,P_s\}\cup (S\cap R),R}(t))=0$
if $t\ge s+|S\cap R|-1$. Hence $h^1(\mathbb {P}^2,\mathcal {I}_{E\cup S}(t))
\le h^1(\mathbb {P}^2,\mathcal {I}_{E'\cup (S\setminus S\cap R),R}(t))$ if $t\ge
s+|S\cap R|-1$.
\end{proof}

\begin{theorem}\label{u0.1}
Fix integers integers $s$ and $d$ such that $2 \le s \le d -1\le  q-2$. Choose
integers $a_1,...,a_s$ such that $0 < a_i \le d+1-i$ for any $i \in \{
1,...,s\}$ and such that $a_1+\cdots +a_s\le 3d-7+s$. Fix $s$ distinct collinear
points $P_1,...,P_s \in 
 X(\mathbb {F}_{q^2})$. Denote by $R$ the line containing the $s \ge 2$ points
$P_1,...,P_s$.
 Take $B:=  X(\mathbb {F}_{q^2})\setminus \{P_1,...,P_s\}$ and set $E:= \sum
_{i=1}^{s} a_iP_i$. 
Then the code $\mathcal{C}=\mathcal{C}(B,d,-E)$ obtained evaluating the vector
space $H^0(X,\Ol_X(d)(-E))$ on the set $B$ has length $n:= q^3+1-s$ and
dimension $k:= \binom{d+2}{2}-\sum_{i=1}^s a_i$. The code $\mathcal{C}^\perp$
has minimum distance $d+2-s$ and its minimum-weight codewords are exactly the
ones whose support, $S$,
 consists of $d+2-s$ points of $B\cap R$. Any $S\subseteq R\cap B$
 with $|S|=d+2-s$ is the support of exactly one (up to multiplication by a
non-zero scalar) minimum-weight codeword.
\end{theorem}

\begin{proof}
The parameters $n$ and $k$ are obvious by Lemma \ref{u5.00} and the inequality
$\deg (E)+|B|
> d\cdot \deg (X)$. Lemma \ref{u5.00} gives also $h^1(\mathbb {P}^2,\mathcal
{I}_E(d)) = 0$. By Lemma \ref{e1} it is enough to prove that $h^1(\mathbb
{P}^2,\mathcal {I}_{E\cup A}(d)) =0$ for any $A\subseteq B$ such that $|A|\le
d+1-s$ and that for any
$S\subseteq B$ such that $|S| =d+2-s$ we have
$h^1(\mathbb {P}^2,\mathcal {I}_{E\cup S}(d)) >0$ if and only if $S\subseteq
R$. 
If $S\subseteq B\cap R$ and $|S| =d+2-s$ then $\deg ((E\cup S)\cap R) =d+2$.
Hence the part (b) of Lemma \ref{u00.01} gives $h^1(\mathbb {P}^2,\mathcal
{I}_{E\cup S}(d))>0$.
So the ``if'' part of the statement is proved.
Now we check the ``only if'' part. Fix
$S\subseteq B$ such that $|S|\le d+2-s$ and $h^1(\mathbb {P}^2,\mathcal
{I}_{E\cup S}(d)) >0$. Since $\deg (E\cup S) \le 4d-5$, we may apply Lemma
\ref{u00.01}.
Let $T$ be a curve arising from the statement of that lemma and set $x:= \deg
(T)$. Define for any $i \in \{ 1,...,s\}$ $E_i:= a_iP_i$ and $e_i := \deg
(E_i\cap T)$. Set $f:= |S\cap T|$. Notice that
if $R$ is not a component of $T$ then $e_i>0$ for at most $x$ indices $i$. Set
$E^{(i)}:= E\setminus E_i$ and $E':= \sum _{i=1}^{s} (a_i-1)P_i$. We will see
$E'$ both as a positive
divisor of $X$ and a zero-dimensional subscheme of $\mathbb {P}^2$. We have
$\deg (E') =
\deg (E)-s$. Consider the scheme $\mbox{Res}_R(E) \subseteq E \subseteq \mathbb
{P}^2$ and observe that since $R$ is transversal to $X$ at each $P_i$, we have
$E' = \mbox{Res}_R(E)$. Keep on mind Lemma \ref{c1} and its proof.

\begin{itemize}
 \item[(i)] In this step we assume that $T$ contains no one of the tangent lines
$L_{X,P_i}$, $1\le i \le s$. Hence $e_i \le x$ for any $i$ (Lemma \ref{u5.0}).
Hence $\deg ((E\cup S)\cap T) \le sx+f$.
First assume $x=1$. We get $|\{P_1,\dots ,P_s\}\cap T|+f \ge d+2$ and hence
$S\subseteq T$,
$P_i\in T$ for any $i$ (i.e. $T=R$)
and $f=d+2-s$, as claimed. Now assume $x=2$ and that $R$
is not a component of $T$. We get $2\cdot 2 +(d+2-s) \ge 2d+2$, a contradiction.
Now
assume that $x=2$ and that $R$ is a component of $T$, say $T =R\cup L'$. If
$|S\cap R| \ge d-s+1$ then $S\subseteq R\cap B$, as claimed; if $S\cap R)\le
d-s$ then $\deg (L'\cap (E'\cup (S\setminus S\cap R)) \ge
d+2+s \ge d+2$. Since $L'\ne L_{X,P_i}$ for any $i$ and $L'\ne R$, we have $\deg
(E'\cap L') \le 1$. Hence $|S| \ge
|S\cap L'|
\ge d+1$, contradiction. If $x=3$ and $R$
is not a component of $T$ then we get $3\cdot 3+d+2-s \le 3d-1$, a
contradiction. Now
assume $x=3$ and $T =R\cup T'$. We may assume
$f \le d+1-s$. Hence $\deg (T'\cap (E'\cup (S\setminus S\cap R)) \ge 3d-(d+1-s)
=2d-1+s$. Since $2d-1+s \ge 2d+2$, we get a contradiction as in the case $x=2$.

\item[(ii)] Now assume that $T$ contains one of the tangent lines $L_{X,P_i}$.
If $x=1$, then $L_{X,P_i} = T$ and hence $\deg (T\cap (E\cup S )) = e_i \le d$,
contradiction. Now assume
$x\ge 2$ and write $T = T'\cup L_{X,P_i}$. We have $T\cap (E\cup S)
=  E_i\sqcup T'\cap (S\cup E^{(i)})$. Assume for the moment that $T'$
contains no tangent line to $X$ at one of the points $P_j$, $j\ne i$. We get
$e_j \le x-1$ for all $j\ne i$ (Lemma \ref{u5.0}). Hence $\deg (T\cap (E\cup S))
=
e_i + \deg (T'\cap (S\cup E^{(i)})) \le e_i+f +(x-1)(s-1)$. First assume $x=2$.
Since $e_i\le d$, we get $f+(s-1) \ge d+2$, a contradiction. Now assume $x=3$,
i.e. $\deg (T')=2$.
First assume $R\subseteq T'$, say $T' = R\cup L''$ with $L''$ a line.
If $|S\cap R| \ge d+2-s$, then we are done. Hence
we may assume $|S\cap R|\le d+1-s$. Hence $\deg (L''\cap( (S\setminus S\cap
R)\cup E)) \ge 3d-e_i -(d+1-s)
\ge d+2$. Since $L''\ne R$ and $L'' \ne L_{X,P_j}$ for any $j$, we have $\deg
(E\cap L'')\le 1$.
Hence $|S| \ge d+1$, a contradiction. Now assume $R\nsubseteq T'$. Since
the points $P_1,\dots ,P_s$ are collinear, Bezout theorem gives
$|T'\cap \{P_1,\dots ,P_s\}| \le 2$. So Lemma \ref{u5.0} implies $\deg
(E^{(i)}\cap T') \le 4$.
We conclude $f \ge 3d-1 -a_i -4 \ge 2d-5 > d+2-s$, a contradiction.
Now assume the existence of an index $j\ne i$ such that $T'$ contains the
tangent line
$L_{X,P_j}$. We have $L_{X,P_j}\cap B=\emptyset$. Hence $f=0$ if $x=2$, a
contradiction.
Aassume $x=3$ and write $T' = L_{X,P_j}\cup L$ with $L$ a line and $|S\cap
L|=f$.
If $L$ is a tangent line to $X$, say at $P$, then either $\deg ((E\cup S) \cap
T)
=e_i+e_j+e_h$ (i.e. $P = P_h$ with $h\in \{1,\dots ,s\}\setminus \{i,j\}
$, or $\deg ((E\cup S)\cap T) \le e_i+e_j+1$
(i.e. $P\notin \{P_1,\dots ,P_s\}$). In any case we get a contradiction, because
$e_i+e_j+e_h \le 3d-1$
when $i, j, h$ are distinct. If $L\ne L_{X,P_h}$ for any $h\notin \{i,j\}$, then
$\deg (((E\setminus (E_i\cup E_j))\cup S)\cap L)
\le f+1$ and equality holds if and only if $P_h\in L$ for some $h\notin
\{i,j\}$. Since $f\le d+1-s$, we
get 
$e_i+e_j+d+2-s \ge 3d$, a contradiction. 
\end{itemize}
\end{proof}

As in the case of 3-point codes, in the following Theorem \ref{m3} we  construct
$s$-point codes with improved parameters. 

\begin{theorem}\label{m3}
Fix integers integers $s$ and $d$ such that $2 \le s \le d -1\le  q-2$. Choose
integers $a_1,...,a_s$ such that $0 < a_i \le d+1-i$ for any $i \in \{
1,...,s\}$, $a_1+\cdots +a_s\le 3d-6$ and $a_i+a_j \le 2d-2$ for any $i\ne j$.
Fix $s$ distinct collinear points $P_1,...,P_s \in 
  X(\mathbb {F}_{q^2})$. Denote by $R$ the line containing the $s \ge 2$ points
$P_1,...,P_s$.
 Take $B':=  X(\mathbb {F}_{q^2})\setminus
R\cap X(\mathbb {F}_{q^2})$ and set $E:= \sum _{i=1}^{s} a_iP_i$. 
Then the code $\mathcal{C}=\mathcal{C}(B',d,-E)$ obtained evaluating the vector
space $H^0(X,\Ol_X(d)(-E))$ on the set $B'$ has length $n:= q^3-q$ and dimension
$k:=\binom{d+2}{2}-\sum_{i=1}^s a_i$. Let $\mathcal{S}$ denote the set of the
lines in $\Pro^2$ different from $R$, containing
one of the points $P_1,\dots ,P_s$ and not tangent to $X$. Let $\mathcal
{S}(d+1)$ be the set of all the subsets
$S\subseteq B'$ such that $|S|=d+1$ and $S\subseteq L$ for some $L\in \mathcal
{S}$. We have $|\mathcal
{S}(d+1)| = s(q^2-1)\binom{q}{d+1}$. The code $\mathcal {C}^{\bot}$ has minimum
distance $d+1$ and for any set $S\in \mathcal {S}(d+1)$
there exists a unique (up to multiplication by a non-zero scalar) minimum-weight
codeword with $S$ as its support. Moreover,
all the minimum-weight codewords of $\mathcal {C}^{\bot}$ are associated to a
unique $S\in \mathcal {S}(d+1)$.
\end{theorem}

\begin{proof}
The parameters $n, k$ of $\mathcal {C}$ are obvious,
because $h^1(\mathbb {P}^2,\mathcal {I}_E(d)) =0$ by Lemma \ref{u5.00} and the
inequality $\deg (E)+|B'|
> d\cdot \deg (X)$ holds. If $A\in \mathcal {S}(d+1)$, then the ``~if~'' part of
Lemma \ref{u00.01} gives
$h^1(\mathbb {P}^2,\mathcal {I}_{E\cup A}(d)) >0$. At this point it is enough to
prove that if $S\subseteq B'$ satisfies $|S| \le d+1$ and $h^1(\mathcal
{I}_{E\cup S}(d)) >0$ then $S\in \mathcal {S}(d+1)$. 
We will make use of the notations $E_i, E^{(i)}, E'$, and so on, introduced in
the proof of Theorem \ref{u0.1}. Keep on mind Lemma \ref{c1} and its proof.
Since $\deg (E\cup S)
=a_1+\cdots +a_s +|S|\le 4d-5$, we may apply Lemma \ref{u00.01} and get a curve
$T\subseteq \mathbb {P}^2$ arising from the statement of that lemma. Set $x:=
\deg (T)\in \{1,2,3\}$. Define $e_i:= \deg (T\cap E_i)$ and $f:= |T\cap S|$.
Let $L\subseteq \mathbb {P}^2$ be any line. We have $\deg (L\cap (E\cup S)) =s$
if $L=R$; $\deg (L\cap (E\cup S)) = a_i$
if $L=L_{X,P_i}$; $\deg (L\cap E) = 1$ if $L\ne R$, $L$ is not tangent to $X$
and $L\cap \{P_1,\dots ,P_s\}\ne
\emptyset$; $L\cap E=\emptyset$ if $L\cap \{P_1,\dots ,P_s\}=\emptyset$.

\begin{itemize}
 \item[(i)] Here we assume that $R$ is not an irreducible component of $T$.
First assume that no $L_{X,P_i}$ is a component of
$T$. Hence
$e_i \le x$ for all $i$ (Lemma \ref{u5.0}). If $x=1$ we get $S\subseteq T$,
$|S|=d+1$ and $L\cap \{P_1,\dots ,P_s\}\ne
\emptyset$ (i.e. $T\in \mathcal {S}$ and $S\in \mathcal {S}(d+1)$), because
$s=\deg ((E\cup S)\cap R) =s <d+2$. Now assume $x=2$ (resp. $x=3$). Since the
points $P_1,\dots ,P_s$ are collinear
and $R\nsubseteq T$, Bezout theorem gives $|\{P_1,\dots ,P_s\}\cap T|\le x$.
Hence
$2\cdot 2 +|S|\ge 2d+2$ (resp. $3\cdot 3+|S| \ge 3d$), a contradiction. Now
assume, say, $L_{X,P_i}\subseteq T$
and set $T = L_{X,P_i}\cup T'$ with $S\cap T = S\cap T'$ and $\deg (T') =x-1$.
If $x=1$ then $\deg ((E\cup S)\cap T) =a_i<d+2$, a contradiction. If $x=2$ then
$a_i + (\deg (E^{(i)}\cup
S)\cap T')
\ge 2d+2 -a_1 \ge d+2$. Hence $S\subseteq T'$, $T'\in \mathcal {S}$ and $S\in
\mathcal {S}(d+1)$. Now assume
$x=3$ and that $T'$ does not contain a line $L_{X,P_j}$ for any $j\ne i$.
We get $\deg (T'\cap E^{(i)})\le \deg (T')^2 =4$. Since $|S| \le d+1$ we
conclude $\deg
(T\cap (E\cup S))
\le a_i+4+d+1 <
3d$, a contradiction. Now assume $L_{X,P_j}\subseteq T$ for some $j\ne i$. We
still have $a_i+a_j+d+1<3d$ (and hence a contradiction),
because
we assumed $a_i+a_j\le 2d-2$ for all $i\ne j$.

\item[(ii)] Now assume that $R$ is an irreducible component
with multiplicity $c\ge 1$ of $T$ and write $T = cR+T_1$ with $\deg (T_1)=x-c$.
Since $R\cap B'=\emptyset$, we have $f = |S\cap T_1|$. Hence
$x>c$. Assume $x=2$, so that $c=1$. We get $\deg (T_1\cap (S\cup E)) \ge 2d+2-s
\ge d+3$, a contradiction.
Now assume $x=3$ and $c=2$. We get $\deg (T_1\cap (E\cup S)) \ge 3d-2s\ge d+2$.
Hence $T_1\in \mathcal {S}$
and $S\in \mathcal {S}(d+1)$. Assume $x=3$ and $c=1$. If $T_1$ contains no
tangent line $L_{X,P_i}$, then
$\deg (T_1\cap E)) \le 2\cdot 2$. Hence $3d \le \deg (T\cap (E\cup S))\le s+
\deg (T_1\cap (E\cup S)) \le s+4+d+1$,
a contradiction. Assume $L_{X,P_i}\subseteq T_1$, say $T_1=L_{X,P_i}\cup T_2$.
Set
$E''[i]:= \sum _{j\ne i} (a_j-1)P_j$. We have $E''[i] = \mbox{Res}_R(E^{(i)})$
and $3d \le \deg (T\cap (E\cup S)) = (a_i-1) +s +\deg (T_2 \cap (S\cup
E''[i]))\le s+a_i-1
+d+2$, a contradiction.
\end{itemize}
\end{proof}

\section{Conclusion}
We introduce classical geometric tools in
the study of geometric codes. The most common methods to investigate
Goppa codes are Weierstrass semigroups, Gr\oe bner bases and other more
computational tools (see the References). Many
important results can be established through these techniques. In this paper we
show how a more geometric approach can reveal an interesting
 bond between the coding-theoretic properties of a Goppa code and the true
geometric attributes of the underlying curve, such as the
geometry of its tangent lines.

\end{document}